\documentclass[a4paper,11pt]{article}
\usepackage{latexsym, amssymb, amsmath, amscd, amsfonts}
\usepackage[mathscr]{eucal}
\usepackage{mathrsfs}
\usepackage{color}
\usepackage{euler,textcomp}
\usepackage[T1]{fontenc}
\usepackage{graphicx}
\usepackage{geometry}
\usepackage{youngtab}

\usepackage{amsthm}

\newtheorem{theorem}{Theorem}[section]

\newtheorem{lemma}[theorem]{Lemma}
\newtheorem{proposition}[theorem]{Proposition}
\newtheorem{corollary}[theorem]{Corollary}

\theoremstyle{definition}
\newtheorem{definition}[theorem]{Definition}

\newtheorem{remark}[theorem]{Remark}
\newtheorem{example}[theorem]{Example}
\newcommand{\V}{\mathcal{V}_k}

\newcommand{\E}{\bold{E}}
\newcommand{\F}{\bold{F}}

\newcommand{\g}{\mathfrak{g}}

\newcommand{\R}{\mathcal{R}}
\newcommand{\M}{\mathcal{M}}

\newcommand{\PP}{\mathcal{P}}

\newcommand{\C}{\mathbb{C}}
\newcommand{\Z}{\mathbb{Z}}
\newcommand{\FF}{B}
\newcommand{\Sy}{\mathfrak{S}}

\input xy
\xyoption{all}
\input{tcilatex}

\usepackage[normalem]{ulem}

\begin{document}

\title{Polynomial functors and categorifications of Fock space}
\author{Jiuzu Hong, Antoine Touz\'e, Oded Yacobi}
\date{}
\maketitle

\begin{abstract}
Fix an infinite  field $k$ of characteristic $p$, and let $\g$ be the  Kac-Moody algebra $\mathfrak{sl}_{\infty}$ if $p=0$ and $\hat{\mathfrak{sl}}_p$ otherwise.  Let $\PP$ denote the category of strict polynomial functors defined over $k$.   We describe a $\g$-action on $\PP$ (in the sense of Chuang and Rouquier) categorifying the Fock space representation of $\g$.  
\end{abstract}

\begin{center}
{\footnotesize \it Dedicated, with gratitude and \\ admiration, to Prof. Nolan Wallach \\on the occasion of his $70^{th}$ birthday.}
\end{center}

\section{Introduction}
Fix an infinite field $k$ of characteristic $p$.  In this work we elaborate on a study, begun in \cite{HY11}, of the relationship between the symmetric groups $\mathfrak{S}_d$, the general linear groups $GL_n(k)$, and the Kac-Moody algebra $\g$, where

\begin{equation*}
\mathfrak{g}=
\left\{
\begin{array}{lr}
\mathfrak{sl}_{\infty}(\mathbb{C}) \text{   if } p=0  &  \\
\widehat{\mathfrak{sl}}_{p}(\mathbb{C}) \text{  if }  p\neq 0
\end{array}
\right.
\end{equation*}
In \cite{HY11} Hong and Yacobi defined a  category $\M$ constructed as an inverse limit of polynomial representations of the general linear groups.  The main result of \cite{HY11} is that $\g$ acts on $\M$ (in the sense of Chuang and Rouquier), and categorifies the Fock space representation of $\g$.  

The result in  \cite{HY11} is motivated by a well known relationship between the basic representation of $\g$ and the symmetric groups.  Let $\R_d$ denote the category of representations of $\Sy_d$ over $k$, and let $\R$ denote the direct sum of categories $\R_d$.  By work going back at least to Lascoux, Leclerc, and Thibon \cite{LLT}, it is known that $\R$ is a categorification of the basic representation of $\g$.  This means that there are exact endo-functors $E_i,F_i:\R \to \R$ ($i \in \Z/p\Z$) whose induced operators on the Grothedieck group give rise to a representation of $\g$ isomorphic to its basic representation. 

Since $\R$ consists of all representations of all symmetric groups, and the representations of symmetric groups and general linear groups are related via Schur-Weyl duality, it is natural to seek a category which canonically considers all polynomial representations of all general linear groups.  This is precisely the limit category of polynomial representations alluded to above.  

The limit category $\M$ is naturally equivalent to the category $\PP$ of ``strict polynomial functors of finite degree'' introduced by Friedlander and Suslin in \cite{FS} (in characteristic zero the category $\PP$ appears in \cite{Mac}).  The objects of $\PP$ are endo-functors on $\V$ (the category of finite dimensional vector spaces over $k$) satisfying natural polynomial conditions, and the morphisms are  natural transformations of functors.

Friedlander and Suslin's original motivation was to study the finite generation of affine group schemes.  Since their landmark work, the theory of polynomial functors has developed in many directions. In algebraic topology, the category $\PP$ is connected to the category of unstable modules over the Steenrod algebra, to the cohomology of the finite linear groups \cite{FFSS,K}, and also to derived functors in the sense of Dold and Puppe \cite{T1}. Polynomial functors are also applied to the cohomology of group schemes. For example the category $\PP$ is used in the study of support varieties for finite group schemes \cite{SFB}, to compute the cohomology of classical groups \cite{T2}, and in the proof of cohomological finite generation for reductive groups \cite{TVdK}. 

The goal of this paper is to develop an explicit connection relating the category of strict polynomial functors to the affine Kac-Moody algebra $\g$.  We describe an action of $\g$ on $\PP$ (in the sense of Chuang and Rouquier), which is completely independent of the results or arguments in \cite{HY11}.  The main advantage of this approach is that the category $\PP$ affords a more canonical setting for the $\g$-categorification.  Indeed, many of the  results obtained in \cite{HY11} have a simple and natural formulation in this setting.  Further, we hope that the ideas presented here will provide new insight to the category polynomial functors.  As an example of this, in the last section of the paper we describe how the categorification theory implies that certain blocks of the category $\PP$ are derived equivalent. 

The category of strict polynomial functors is actually defined over arbitrary fields, but the general definition given in [FS] is more involved than the one we use (the problem comes from the fact that different polynomials might induce the same function over finite fields).  All our results remain valid in this general context, but we have opted to work over an infinite field to simplify the exposition.  In addition, we assume in our main theorem that $p\neq2$.  The theorem is valid also for $p=2$, but including this case would complicate our exposition.    

In the sequel to this work we continue the study of $\PP$ from the point of view of higher representation theory \cite{HY2}.  We show that Khovanov's category $\mathcal{H}$ naturally acts on $\PP$, and this gives a categorification of the Fock space representation of the Heisenberg algebra when $char(k)=0$.  When $char(k)>0$ the commuting actions of $\g'$ (the derived algebra of $\g$) and the Heisenberg algebra are also categorified.  Moreover, we formulate Schur-Weyl duality as a functor from $\PP$ to the category of linear species.  The category of linear species is known to carry actions of $\g$ and the Heisenberg algebra.  We prove that Schur-Weyl duality is a tensor functor which is a morphism of these categorifications structures.

\section{Type A Kac-Moody algebras}
\label{Fock space}
 Let $\mathfrak{g}$ denote the following Kac-Moody algebra (over $\C$):
\begin{equation*}
\mathfrak{g}=
\left\{
\begin{array}{lr}
\mathfrak{sl}_{\infty} \text{   if } p=0  &  \\
\widehat{\mathfrak{sl}}_{p} \text{  if }  p>0
\end{array}
\right.
\end{equation*}
By definition, the Kac-Moody algebra $\mathfrak{sl}_{\infty}$ is associated to the Dynkin diagram:
$$
\xymatrix{
\cdots \ar@{-}[r] & \bullet \ar@{-}[r] & \bullet \ar@{-}[r] & \bullet\ar@{-}[r] & \bullet\ar@{-}[r] & \cdots}
$$
while the Kac-Moody algebra $\widehat{\mathfrak{sl}}_{p}$ is associated to the diagram with $p$ nodes:
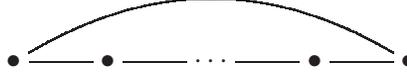
\linebreak
\linebreak
$$
\xymatrix{
\bullet \ar@{-}@/^2pc/[rrrr] \ar@{-}[r] & \bullet \ar@{-}[r] & \cdots \ar@{-}[r] & \bullet\ar@{-}[r] & \bullet
}
$$
The Lie algebra $\g$ has standard Chevalley generators $\{e_i,f_i\}_{i \in \Z /p\Z}$. Here, and throughout, we identify $\Z/p\Z$ with the prime subfield of $k$.  For the  precise relations defining $\g$ see e.g. \cite{Kac}.   

We let $Q$ denote the root lattice and $P$ the weight lattice of $\g$.  Let $\{\alpha_i : i \in \Z/p\Z \}$ denote the set of simple roots, and $\{h_i : i \in \Z/p\Z \}$ the simple coroots. The cone of dominant weights is denoted $P_+$ and denote the fundamental weights $\{\Lambda_i:i \in \Z/p\Z\}$, i.e. $\left< h_i , \Lambda_j \right> = \delta_{ij}$.  When $p >0 $ the Cartan subalgebra of $\g$ is spanned by the $h_i$ along with an element $d$.  In this case we also let $\delta=\sum_{i}\alpha_i$; then $\Lambda_0,...,\Lambda_{p-1},\delta$ form a $\Z$-basis for $P$.  When $p=0$ the fundamental weights are a $\Z$ basis for the weight lattice.     

Let $\mathfrak{S}_n$ denote the symmetric group on $n$ letters.  $\mathfrak{S}_n$ acts on the polynomial algebra $\Z[x_1,...,x_n]$ by permuting variables, and we denote by $B_n=\Z[x_1,...,x_n]^{\mathfrak{S}_n}$  the polynomials invariant under this action.  There is a natural projection
$
B_n \twoheadrightarrow B_{n-1}
$ 
given by setting the last variable to zero.  Consequently, the rings $B_{n}$ form a inverse system; let $B_\Z$ denote the subspace of finite degree elements in the inverse limit
$\varprojlim B_n$.  This is the \emph{algebra of symmetric functions} in infinitely many variables $\{x_1,x_2,... \}$.  Let $B=B_\Z\otimes_\Z\mathbb{C}$ denote the  \emph{(bosonic) Fock space}.

The algebra $B_\Z$ has many  well-known bases.  Perhaps the nicest is the basis of \emph{Schur functions} (see e.g. \cite{Mac}).  Let $\wp$ denote the set of all partitions, and for $\lambda\in\wp$ let $s_\lambda\in B_\Z$ denote the corresponding Schur function.  
These form a $\Z$-basis of the algebra of symmetric functions:
$$
B_\Z=\bigoplus_{\lambda\in\wp}\Z s_{\lambda}.
$$  

Let us review some combinatorial notions related to Young diagrams. Firstly, we identify partitions with their Young diagram (using English
notation).  For example the partition $(4,4,2,1)$ corresponds to the
diagram
 $$
 \yng(4,4,2,1)
 $$
The \emph{content} of a box in position $(k,l)$ is the integer $l-k
\in \Z/p\Z$.  Given $\mu,\lambda \in \wp$, we write
$
\xymatrix{ \mu \ar[r] & \lambda}
$
if $\lambda$ can be obtained from $\mu$ by adding some box.  If the
arrow is labelled $i$ then $\lambda$ is obtained from $\mu$ by
adding a box of content $i$ (an $i$-box, for short).  For instance,
if $m=3$, $\mu=(2)$ and $\lambda=(2,1)$ then
$
\xymatrix{
 \mu \ar[r]^{2} & \lambda}.
$
An $i$-box of $\lambda$ is \emph{addable} (resp. \emph{removable})
if it can be added to (resp. removed from) $\lambda$ to obtain
another partition. 

Of central importance to us is the \emph{Fock space} representation
of $\g$ on $B$ (or $B_\Z$).  The action of $\g$ on $B$ is given by the following formulas:
$
e_i.s_{\lambda}=\sum s_{\mu}
$,
the sum over all $\mu$ such that $\xymatrix{ \mu \ar[r]^{i} &
\lambda}$, and
 $
f_i.s_{\lambda}=\sum s_{\mu}
$,
the sum over all $\mu$ such that $\xymatrix{ \lambda \ar[r]^{i} &
\mu}$.     
Moreover,  $d$ acts on $s_\lambda$ by $m_0(\lambda)$, where $m_0(\lambda)$ is the number of boxes of content zero in $\lambda$.  These equations define an integral  representation of
$\mathfrak{g}$.(see e.g. \cite{LLT}).

Note that $s_{\emptyset}$ is a highest weight vector of highest weight $\Lambda_0$. We note also that the standard basis of $\FF$ is a weight basis.  Let $m_i(\lambda)$ denote the number of $i$-boxes of $\lambda$.  Then $s_\lambda$ is of weight $wt(\lambda)$, where 
\begin{equation}
\label{weight}
wt(\lambda)={\Lambda_0-\sum_i m_i(\lambda) \alpha_i.}
\end{equation}

For a $k$-linear abelian category $\mathcal{C}$, let $K_0(\mathcal{C})$ denote the Grothendieck group of $\mathcal{C}$, and let $K(\mathcal{C})$ denote the complexification of $K_0(\mathcal{C})$.  If $A\in\mathcal{C}$ we let $[A]$ denote its image in $K_0(\mathcal{C})$. Simiarly, for an exact functor $F:\mathcal{C}\to\mathcal{C}'$ we let $[F]:K_0(\mathcal{C})\to K_0(\mathcal{C}')$ denote the induced operator on the Grothendieck groups.  Slightly abusing notation, the complexification of $[F]$ is also denoted by $[F]$.

\section{The definition of $\g$-categorification}

``Higher representation theory'' of $\g$ concerns the action of $\g$ on categories rather than on vector spaces.  At the very least, an action of $\g$ on a $k$-linear additive category $\mathcal{C}$ consists of the data of exact endo-functors $E_i$ and $F_i$ on $\mathcal{C}$ (for $i \in \Z/p\Z$), such that $\g$ acts on $K(\mathcal{C})$ via the assignment $e_i \mapsto [E_i]$ and $f_i \mapsto [F_i]$.  For instance, if $i$ and $j$ are not connected in the Dynkin diagram of $\g$ (i.e. $[e_i,f_j]=0$), then we require that $[[E_i],[F_j]]=0$ in $End(K(\mathcal{C}))$.   This is known as a ``weak categorification''.

This notion is qualified as ``weak'' because the relations defining $\g$, such as $[e_i,f_j]=0$, are not lifted to the level of categories.  A stronger notion of categorification would require isomorphisms of functors lifting the relations of $\g$, e.g. functorial isomorphisms $E_i \circ F_j \simeq F_j \circ E_i$.  Moreover, these isomorphisms need to be compatible in a suitable sense.  Making these ideas precise  leads to an enriched theory, which introduces new symmetries coming from an affine Hecke algebra.   

To give the definition of $\g$-categorification we use here,  due to Chuang and Rouquier (a related formulation appears in the works of Khovanov and Lauda \cite{KL}), we first introduce the relevant Hecke algebra. 

\begin{definition}
Let $DH_n$ be the \emph{degenerate affine Hecke algebra} of $GL_n$.  As an abelian group
$$
DH_n=\Z[y_1,...,y_n]\otimes \Z \mathfrak{S}_n.
$$
The algebra structure is defined as follows: $\Z[y_1,...,y_n]$ and $\Z \mathfrak{S}_n$ are subalgebras, and the following relations hold between the generators of these subalgebras:
$$
\tau_iy_j=y_j\tau_i \text{ if }|i-j| \geq 1
$$
and
\begin{equation}
\label{HeckeRelation}
\tau_iy_{i+1}-y_{i}\tau_i=1
\end{equation}
(here $\tau_1,...,\tau_{n-1}$ are the simple generators of $\Z \mathfrak{S}_n$).
\end{definition}

\begin{remark}
One can replace Relation (\ref{HeckeRelation}) by \begin{equation}
\label{HeckeRelation'}
\tau_iy_{i}-y_{i+1}\tau_i=1.
\end{equation}
These two presentations are equivalent; the isomorphism is given by: 
$$\tau_i\mapsto\tau_{n-i}, y_i\mapsto y_{n+1-i}. $$ 
\end{remark}

\begin{definition}
\label{DefinitionCategorification}[Definition 5.29 in \cite{R}]
Let $\mathcal{C}$ be an abelian $k$-linear category.  A \emph{$\g$-categorification} on $\mathcal{C}$ is the data of:
\begin{enumerate}
\item An adjoint pair $(E,F)$ of exact functors $\mathcal{C} \rightarrow \mathcal{C},$
\item morphisms of functors $X \in End(E)$ and $T \in End(E^2), $ and
\item a decomposition $\mathcal{C}=\bigoplus_{\omega \in P} \mathcal{C}_{\omega}$.
\end{enumerate}
Let $X^{\circ} \in End(F)$ be the endomorphism of $F$ induced by adjunction.  Then given $a \in k$ let $E_a$ (resp. $F_a$) be the generalized $a$-eigensubfunctor of $X$ (resp. $X^{\circ}$) acting on $E$ (resp. $F$).  We assume that
\begin{enumerate}
\item[4.] $E=\bigoplus_{i \in \Z/p\Z} E_i$,
\item[5.] the action of $\{ [E_i],[F_i]\}_{i \in \Z/p\Z}$ on $K(\mathcal{C})$ gives rise to an integrable representation of $\g,$
\item[6.] for all $i$, $E_i(\mathcal{C}_{\omega}) \subset \mathcal{C}_{\omega+\alpha_i}$ and $F_i(\mathcal{C}_{\omega}) \subset \mathcal{C}_{\omega-\alpha_i}$,
\item[7.] the functor $F$ is isomorphic to the left adjoint of $E$, and
\item[8.] the degenerate affine Hecke algebra $DH_n$ acts on $End(E^n)$ via
\begin{equation}
\label{EqXX}
y_i \mapsto E^{n-i}XE^{i-1} \text{ for }1 \leq i \leq n,
\end{equation}
and
\begin{equation}
\label{EqTT}
\tau_i \mapsto E^{n-i-1}TE^{i-1} \text{ for }1 \leq i \leq n-1.
\end{equation}
\end{enumerate}

\end{definition}
\begin{remark}
\label{Pedantry}
The  definition (cf. Definition 5.29 in \cite{R}) uses Relation (\ref{HeckeRelation}). For our purposes we use Relation (\ref{HeckeRelation'}).  On the representations of the symmetric groups (the main example considered in \cite[Section 3.1.2]{CR}) another variant of Relation (\ref{HeckeRelation'}) is used .  
\end{remark}
\begin{remark}
To clarify notation, the natural endomorphism $y_i$ of $E^n$ assigns to $M \in \mathcal{C}$ an endomorphism of $E^n(M)$ as follows: apply the functor $E^{n-i}$ to the morphism $$X_{E^{i-1}(M)}:E^i(M) \rightarrow E^i(M).$$
\end{remark}

The functorial isomorphisms lifting the defining relations of $\g$ are constructed from the data of $\g$-categorification.  More precisely, the adjunctions between $E$ and $F$ and the functorial morphisms $X$ and $T$ are introduced precisely for this purpose.  The action of $DH_n$ on $End(E^n)$ in part (8) of Definition \ref{DefinitionCategorification} is needed in order to express the compatibility between the functorial isomorphisms.  See \cite{R} for details.    


\section{Polynomial functors}

\subsection{The category $\PP$}
Our main goal in this paper is to define a $\g$-categorification on the category $\PP$ of \emph{strict polynomial functors of finite degree}, and show that this categorifies the Fock space representation of $\g$. In this section we define the category $\PP$ and recall some of its basic features.   

Let $\V$ denote the category of finite dimensional vector spaces over $k$.  For $V,W \in \V$, \emph{polynomial maps} from $V$ to $W$ are by definition elements of $S(V^{*})\otimes W$, where $S(V^*)$ denotes the symmetric algebra of the linear dual of $V$.  Elements of $S^{d}(V^{*})\otimes W$ are said to be \emph{homogeneous} of degree $d$.
\begin{definition}
  The objects of the category $\PP$ are functors $M:\V \rightarrow \V$ that satisfy the following properties:
\begin{enumerate}
\item for any $V,W \in \V$, the map of vector spaces
$$
Hom_{k}(V,W) \rightarrow Hom_{k}(M(V),M(W))
$$
is polynomial, and
\item the degree of the map 
$$
End_{k}(V) \rightarrow End_{k}(M(V))
$$
is bounded with respect to $V \in \V$.
\end{enumerate}
The morphisms in $\PP$ are natural transformations of functors.  For $M \in \PP$ we denote by $1_M\in Hom_{\PP}(M,M)$ the identity natural transformation. \end{definition}

Let $I \in \PP$ be the identity functor from $\V$ to $\V$ and let $k\in \PP$ denote the constant functor with value $k$. 
Tensor products in $\V$ define a symmetric  monoidal structure $\otimes$ on $\PP$, with unit $k$. The category $\PP$ is abelian.

Let $M \in \PP$ and $V \in \V$.  By functoriality $M(V)$ carries a polynomial action of the linear algebraic group $GL(V)$.  We denote this representation by $\pi_{M,V}$, or by $\pi$ when the context is clear:
$$\pi_{M,V}:GL(V)\to GL(M(V))\;.$$ 
Similarly, a morphism $\phi:M\to N$ induces a $GL(V)$-equivariant map $\phi_ V:M(V)\to N(V)$. Thus, evaluation on $V$ yields a functor from
$\PP$ to  $Pol(GL(V))$, the category of polynomial representations of $GL(V)$.

\subsection{Degrees and weight spaces}

The \emph{degree} of a functor $M\in\PP$ is the upper bound of the degrees of the polynomials $End_k(V)\to End_k(M(V))$ for $V\in\V$. For example, the functors of degree zero are precisely the functors $\V\to\V$ which are isomorphic to constant functors. 
A functor $M\in\PP$ is \emph{homogeneous of degree $d$} if all the polynomials $End_k(V)\to End_k(M(V))$ are homogeneous polynomials of degree $d$.
 
For $M\in\PP$, $GL(k)$ acts on $M(V)$ by the formula, 
$$\lambda\cdot m=\pi_{M,V}(\lambda 1_V)(m), \text{ for $\lambda\in GL(k)$ and $m\in M(V)$}\;.$$ 
This action is a polynomial action of $GL(k)$, so $M(V)$ splits as a direct sum of weight spaces
$$M(V)=\bigoplus_{d\ge 0}M(V)_d;, $$
where
$$M(V)_d=\{m\in M(V) : \lambda\cdot m = \lambda^dm  \}\;.$$
Moreover, if $f:V\to W$ is a linear map, it commutes with homotheties, so $M(f)$ is $GL(k)$-equivariant. Hence $M(f)$ preserves weight spaces, and we denote by $M(f)_d$ its restriction to the $d$-th weight spaces.

So we can define a strict polynomial functor $M_d$ by letting 
$$M_d(V)= M(V)_d\;, M_d(f)=M(f)_d;.$$
A routine check shows that $M_d$ is homogeneous of degree $d$.
Thus, any functor $M$ decomposes as a finite direct sum of homogeneous functors $M_d$ of degree $d$. 
Similarly, a morphism $\phi:M\to N$ between strict polynomial functors preserves weight spaces. So it decomposes as a direct sum of morphisms of homogeneous functors $\phi_d:M_d\to N_d$. This can be formulated by saying that the category $\PP$ is the direct sum of its subcategories $\PP_d$ of homogeneous functors of degree $d$:
\begin{equation}
\label{Pdecomp}
\PP=\bigoplus_{d\ge 0} \PP_d\;.
\end{equation} 

If $M\in\PP$, we define its \emph{Kuhn dual} $M^\sharp\in\PP$ by  $M^\sharp(V)=M(V^*)^*$, where `$*$' refers to $k$-linear duality in the category of vector spaces. Since $(M^{\sharp})^\sharp\simeq M$, duality yields an equivalence of categories \cite[Prop 2.6]{FS}:
$$^\sharp:\PP\xrightarrow[]{\simeq} \PP^{\mathrm{op}}\;.$$
A routine check shows that $^\sharp$ respect degrees, i.e. $M^\sharp$ is homogeneous of degree $d$ if and only if $M$ also is.

The following theorem, due to Friedlander and Suslin \cite{FS}, shows  the categories $\PP_d$ are a model for the stable categories of homogeneous polynomial $GL_n(k)$-modules of degree $d$.  Let $Pol_d(GL(V))$ denote the category of polynomial representations of $GL(V)$ of degree $d$.
\begin{theorem}\label{thm-FS}
Let $V\in\V$ be a $k$-vector space of dimension $n\ge d$. The functor induced by evaluation on $V$:
$$\PP_d\to Pol_d(GL(V)),$$
is an equivalence of categories.
\end{theorem}

As a consequence of Theorem \ref{thm-FS}, we obtain that strict polynomial functors are noetherian objects.
\begin{corollary}\label{cor-finite-filtration}
Let $M\in\PP$. Assume that you have an increasing sequence of subfunctors of $M$:
$$M^0\subset M^1\subset \dots\subset M^i\subset\dots \;.$$
Then there exists an integer $N$ such that for all $n\ge N$, $M^n=M^N$.
\end{corollary}

\begin{remark}\label{rk-tensors}
Let $\otimes^d$ denote the $d$-th tensor product functor, which sends $V\in\V$ to $V^{\otimes d}\in\V$. Then $\otimes^d\in\PP_d$.
Let $\lambda$ be a tuple of nonnegative integers summing to $d$, and let $\mathfrak{S}_\lambda\subset \mathfrak{S}_d$ denote the associated Young subgroup. 
We denote by $\Gamma^\lambda$ the subfunctor of $\otimes^d$ defined by $\Gamma^{\lambda}(V)=(V^{\otimes d})^{\mathfrak{S}_\lambda}$. 
It is proved in \cite{FS} that the functors  $\Gamma^\lambda$, $\lambda\in\wp$ a partition of $d$, form a (projective) generator of $\PP_d$.
In other words, the objects $M\in\PP_d$ are exactly the functors $M:\V\to\V$ which can be obtained as subquotients of a direct sum of a finite number of copies of the $d$-th tensor product functor $\otimes^d$. 
\end{remark}

\subsection{Recollections of Schur and Weyl functors}
In this section we introduce Schur functors and Weyl functors. These strict polynomial functors are the functorial version of the Schur modules and the Weyl modules, and they were first defined in \cite{ABW}. 

Let $\wp_d$ denote the partitions of $d$. For $\lambda\in\wp_d$ let $\lambda^{\circ}$ denote the conjugate partition. We define a map $d_\lambda$ as the composite:
$$d_\lambda\,:\;\Lambda^{\lambda^{\circ}_1}\otimes\cdots\otimes \Lambda^{\lambda^{\circ}_n}\hookrightarrow \otimes^d \xrightarrow[]{\sigma_\lambda}\otimes^d\twoheadrightarrow S^{\lambda_1}\otimes\cdots\otimes S^{\lambda_m}\;. $$
Here the first map is the canonical inclusion and the last one is the canonical epimorphism. The middle map is the isomorphism of $\otimes^d$ which maps $v_1\otimes\cdots v_d$ onto $v_{\sigma_{\lambda}(1)}\otimes\cdots\otimes v_{\sigma_{\lambda}(d)}$, where $\sigma_\lambda\in\mathfrak{S}_d$ is the permutation defined as follows.  Let $t_\lambda$ be the  Young tableaux with standard filling: $1,...,\lambda_1$ in the first row, $\lambda_1+1,...,\lambda_2$ in the second row, and so forth.  Then $\sigma_\lambda$, in one-line notation, is the row-reading of the conjugate tableaux $t^{\circ}_\lambda$.  For example, if $\lambda=(3,1)$, then, $\sigma_{\lambda}=1423$, the permutation mapping $1\mapsto 1, 2\mapsto 4, 3\mapsto 2$, and $4\mapsto 3$.  
\begin{definition}
Let $\lambda \in \wp_d$.
\begin{enumerate}
\item The \emph{Schur functor} $S_\lambda\in\PP_d$ is the image of $d_\lambda$. 
\item The \emph{Weyl functor} $W_\lambda$ is defined by duality $W_\lambda:=S_\lambda^\sharp$.   
\item Let $L_\lambda$ be the socle of the functor $S_\lambda$.
 \end{enumerate}
\end{definition}

\begin{remark}
In \cite[def. II.1.3]{ABW}, Schur functors are defined in the more general setting of `skew partitions' $\lambda/\alpha$, (i.e. pairs of partitions $(\lambda,\alpha)$ with $\alpha\subset \lambda$), and over arbitrary commutative rings. They denote  Schur functors by $L_{\lambda^\circ}$, but we prefer to reserve this notation for simple objects in $\PP_d$. 
\end{remark}

The following statement makes the link between Schur functors and induced modules (also called costandard modules, or Schur modules)  and between Weyl functors and Weyl modules (also called standard modules or Verma modules).

\begin{proposition}\label{prop-corr} Let $\lambda\in\wp_d$.
\begin{itemize}
\item[(i)] There is an isomorphism of $GL(k^n)$-modules $S_\lambda(k^n)\simeq H^0(\lambda)$, where $H^0(\lambda)=ind_{B^-}^{GL(k^n)}(k^\lambda)$ is the induced module from \cite[II.2]{Jantzen}.
\item[(ii)] There is an isomorphism of $GL(k^n)$-modules $W_\lambda(k^n)\simeq V(\lambda)$, where $$V(\lambda)=H^0(-w_0\lambda)^*$$ is the Weyl module from \cite[II.2]{Jantzen}.
\end{itemize}
\end{proposition}
\begin{proof}
We observe that (ii) follows from (i). Indeed, we know that $V(\lambda)$ is the transpose dual of $H^0(\lambda)$, and evaluation on $k^n$ changes the duality $^\sharp$ in $\PP$ into the transpose duality. To prove (ii), we refer to \cite{Martin}. The Schur module $M(\lambda)$ defined in \cite[Def 3.2.1]{Martin} is isomorphic to $H^0(\lambda)$ (this is a theorem of James, cf. \cite[Thm 3.2.6]{Martin}). Now, using the embedding of $M(\lambda)$ into $S^{\lambda_1}(k^n)\otimes\cdots \otimes S^{\lambda_m}(k^n)$ of \cite[Example (1) p.73]{Martin}, and \cite[Thm II.2.16]{ABW}, we get an isomorphism $S_\lambda(k^n)\simeq M(\lambda)$. 
\end{proof}

The following portemanteau theorem collects some of the most important properties of the functors $S_\lambda$, $W_\lambda$, $L_\lambda$, $\lambda\in\wp_d$.

\begin{theorem}\label{portemanteau}

\begin{enumerate}
\item[(i)] The functors $L_\lambda$, $\lambda\in\wp_d$ form a complete set of representatives for the isomorphism classes of irreducible functors of $\PP_d$.
\item[(ii)] Irreducible functors are self dual: for all $\lambda\in\wp_d$, $L_\lambda^\sharp\simeq L_\lambda$.
\item[(iii)] For all $\lambda\in\wp_d$, the $L_\mu$ which appear as composition factors in $S_\lambda$ satisfy $\mu\le \lambda$, where $\le$ denotes the lexicographic order. Moreover, the multiplicity of $L_\lambda$ in $S_\lambda$ is one.
\item[(iv)] For all $\lambda,\mu\in\wp_d$, 
$$Ext^i_\PP(W_\mu,S_\lambda)= \left\lbrace
{\begin{array}{l}
k\text{ if $\lambda=\mu$ and $i=0$,}\\
0\text{ otherwise.}
\end{array}
}\right.$$
\end{enumerate}
\end{theorem}
\begin{proof}
All these statements have functorial proofs, but for sake of brevity we shall use proposition \ref{prop-corr}, together with the fact that evaluation on $V$ for $\dim V\ge d$ is an equivalence of categories. Thus, (i) follows from \cite[Thm. 3.4.2]{Martin}, (ii) follows from \cite[Thm. 3.4.9]{Martin}, (iii) follows from \cite[Thm. 3.4.1(iii)]{Martin}. Finally, (iv) follows from \cite[Prop. 4.13]{Jantzen} and \cite[Cor. 3.13]{FS}.
\end{proof}

\begin{corollary}
\label{LemmaKgrp}
The equivalence classes of the Weyl functors $[W_\lambda]$ for $\lambda\in\wp$ form a basis of $K(\PP)$.
\end{corollary}
\begin{proof}
Order $\wp$ by the lexicographic order, denoted $\leq$.  By parts (ii) and (iii) of Theorem \ref{portemanteau}, the multiplicity of $L_{\lambda}$ in $W_{\lambda}$ is one, and all other simple objects appearing as composition factors in $W_\lambda$ are isomorphic to $L_\mu$, where $\mu\le \lambda$.  Form the matrix of the map given by $[L_\lambda]\mapsto [W_\lambda]$ in the basis $[L_\lambda]_{\lambda\in\wp}$ (ordered by $\leq$). This is a lower triangular matrix, with $1$'s on the diagonal. Hence it is invertible and we obtain the result.
\end{proof}

\begin{corollary}\label{lm-dual}
The map $K(\PP)\to K(\PP)$ given by $[M]\mapsto [M^\sharp]$ is the identity.
In particular, for all $\lambda\in\wp$, $[W_\lambda]=[S_\lambda]$.
\end{corollary}
\begin{proof}
By Theorem \ref{portemanteau}(ii) simple functors are self-dual. Whence the result.
\end{proof}

\subsection{Polynomial bifunctors}\label{SectionBi}
We shall also need the category $\PP^{[2]}$ of \emph{ strict polynomial bi-functors}. The objects of $\PP^{[2]}$ are functors $B:\V\times\V\to \V$ such that for every $V\in\V$, the functors $B(\cdot,V)$ and $B(V,\cdot)$ are in $\PP$ and their degrees are bounded with respect to $V$. Morphisms in $\PP^{[2]}$ are natural
transformations of functors. The following example will be of particular interest to us.

\begin{example}
Let $M\in\PP$. We denote by $M^{[2]}$ the bifunctor:
$$\begin{array}{cccc}M^{[2]}:&\V\times\V&\to &\V\\
& (V,W)&\mapsto & M(V\oplus W)\\
& (f,g)&\mapsto & M(f\oplus g)
\end{array}.$$
Mapping $M$ to $M^{[2]}$ yields a functor:
$\PP\to \PP^{[2]}$.
\end{example}

If $B\in\PP^{[2]}$ and $(V,W)$ is a pair of vector spaces, then functoriality endows $B(V,W)$ with a polynomial $GL(V)\times GL(W)$-action, which we denote by $\pi_{B,V,W}$ (or simply by $\pi$ if the context is clear):
$$\pi_{B,V,W}:GL(V)\times GL(W)\to GL(B(V,W))\;.$$
Evaluation on a pair $(V,W)$ of vector spaces yields a functor from $\PP^{[2]}$ to  $Rep_{pol}(GL(V)\times GL(W))$.

A bifunctor $B$ is \emph{homogeneous of bidegree} $(d,e)$ if for all $V\in\V$, $B(V,\cdot)$ (resp. $B(\cdot,V)$) is a homogeneous strict polynomial functor of degree $d$, (resp. of degree $e$). The decomposition of strict polynomial functors into a finite direct sums of homogeneous functors generalizes to bifunctors. Indeed, if $B\in\PP^{[2]}$ the vector space $B(V,W)$ is endowed with a polynomial action of $GL(k)\times GL(k)$ defined by:
$$(\lambda,\mu)\cdot m= \pi_{B,V,W}(\lambda 1_V,\mu 1_W)(m)\;, $$
and pairs of linear maps $(f,g)$ induce $GL(k)\times GL(k)$-equivariant morphisms $B(f,g)$. So for $i,j\ge 0$ we can use the $(i,j)$ weight spaces with respect to the action of $GL(k)\times GL(k)$ to define bifunctors $B_{i,j}$, namely
$$B_{i,j}(V,W)=\{m\in B(V,W): (\lambda,\mu)\cdot m = \lambda^i\mu^jm\}$$
and $B_{i,j}(f,g)$ is the restriction of $B(f,g)$ to the $(i,j)$-weight spaces. Functors $B_{i,j}$ are homogenous of bidegree $(i,j)$ and $\PP^{[2]}$ splits as the direct sum of its full subcategories $\PP^{[2]}_{i,j}$ of homogeneous bifunctors of bidegree $(i,j)$. If $B\in\PP^{[2]}$, we denote by $B_{*,j}$ the direct sum:
\begin{equation}
\label{Bdecomp}
B_{*,j}=\bigoplus_{i\ge 0}B_{i,j}
\end{equation}
Note that we have also a duality for bifunctors 
$$^\sharp:\PP^{[2]}\xrightarrow[]{\simeq} \PP^{[2]\,\mathrm{op}}$$
which sends $B$ to $B^\sharp$, with $B^\sharp(V,W)=B(V^*,W^*)^*$, and which respects the bidegrees. 

The generalization of these ideas to the category of strict polynomial \emph{tri-functors} of finite degree $\PP^{[3]}$, which contains the tri-functors $M^{[3]}:(U,V,W)\mapsto M(U\oplus V\oplus W)$, and so on, is straight-forward.

We conclude this section by introducing a construction of new functors in $\PP$ from old ones that will be used in the next section.  Let $M\in\PP$ and consider the functor $M^{[2]}(\cdot,k)\in\PP$. By (\ref{Bdecomp}) we have a decomposition
$$
M^{[2]}(\cdot,k)=\bigoplus_{i\geq0} M^{[2]}_{*,i}(\cdot,k).
$$
In other words, $M^{[2]}_{*,i}(V,k)$ is the subspace of weight $i$ of $M(V\oplus k)$ acted on by $GL(k)$ via the composition 
$$GL(k)=1_V\times GL(k)\hookrightarrow GL(V\oplus k)\xrightarrow[]{\pi_{M,V\oplus k}} GL(M(V\oplus k))\;.$$
Since evaluation on $V\oplus k$ as well as taking weight spaces are exact, the assignment $M \mapsto M^{[2]}_{*,i}(\cdot,k)$ defines an exact endo-functor on $\PP$.

\section{Categorification Data}
Having defined the notion of $\g$-categorification and the category $\PP$, we are now ready to begin the task of defining a $\g$-categorification on $\PP$.  The present section is devoted to introducing the necessary data to construct the categorification (cf.  items (1)-(3) of Definition \ref{DefinitionCategorification}.  The following section will be devoted to showing that this data satisfies the required properties (cf. items (4)-(8) of Definition \ref{DefinitionCategorification}).

\subsection{The functors $\E$ and $\F$}
Define $\E,\F:\PP\to\PP$ by 
\begin{eqnarray*}
\E(M)&=&M^{[2]}_{*,1}(\cdot,k) \\ \F(M)&=&M\otimes I
\end{eqnarray*}
for $M\in\PP$.  These are exact functors ($\F$ is clearly exact; for the exactness of $\E$ see the last paragraph of Section \ref{SectionBi}).  We prove that $\E$ and $\F$ are bi-adjoint.
\begin{proposition}\label{first-adjunction}
The pair $(\F,\E)$ is an adjoint pair, i.e. we have an isomorphism, natural with respect to $M,N\in\PP$:
$$\beta:Hom_\PP(\F(M),N)\simeq Hom_\PP(M,\E(N))\;. $$
\end{proposition}
\begin{proof}
We shall use the category $\PP^{[2]}$ of strict polynomial bifunctors. There are functors:
\begin{align*}
&\boxtimes:\PP\times\PP\to \PP^{[2]} &&\otimes:\PP\times\PP\to \PP\\
&\Delta:\PP^{[2]}\to \PP && ^{[2]}:\PP\to \PP^{[2]}
\end{align*}
respectively given by 
\begin{align*}
&M\boxtimes N(V,W)=M(V)\otimes N(W) &&M\otimes N(V)= M(V)\otimes N(V)\\
&\Delta B (V)=B(V,V) && M^{[2]}(V,W)=M(V\oplus W)
\end{align*}
We observe that $\Delta(M\boxtimes N)= M\otimes N$. Moreover, we know (cf. \cite[Proof of Thm 1.7]{FFSS} or \cite[Lm 5.8]{T2}) that $\Delta$ and $^{[2]}$ are bi-adjoint.

Now we are ready to prove the adjunction isomorphism. We have the following natural isomorphisms:
\begin{align*}
Hom_\PP(\F(M),N)&=Hom_\PP(M\otimes I,N)\\
&\simeq Hom_{\PP^{[2]}}(M\boxtimes I, N^{[2]})\\
&\simeq Hom_\PP(M(\cdot),Hom_\PP(I(*),N(\cdot \oplus *)))
\end{align*}
Here $Hom_\PP(I(*),N(\cdot \oplus *))$ denotes the polynomial functor which assigns to $V\in\V$ the vector space $Hom_\PP(I,N(V\oplus *))$.  By Yoneda's Lemma \cite[Thm 2.10]{FS}, for any $F\in\PP$, $Hom_{\PP}(I,F)\simeq F(k)$ if $F$ is of degree one, and zero otherwise. In particular, $Hom_\PP(I,N(V\oplus *))\simeq N(V\oplus k)_1=\E(N)(V)$.  Hence, $Hom_\PP(I(*),N(\cdot \oplus *)) \simeq \E(N)$ and we conclude that there is a natural isomorphism:
$$
Hom_\PP(\F(M),N)\simeq Hom_\PP(M,\E(N)).
$$    
\end{proof}

We are now going to derive the adjunction $(\E,\F)$ from proposition \ref{first-adjunction} and a duality argument. 
The following lemma is an easy check.

\begin{lemma}
For all $M\in\PP$, we have isomorphisms, natural with respect to $M$:
$$\F(M)^\sharp\simeq \F(M^\sharp)\;,\, \E(M)^\sharp\simeq \E(M^\sharp)\;.$$
\end{lemma}
\begin{proof}
We have an isomorphism:
$$\F(M)^\sharp= (M\otimes I)^\sharp\simeq M^\sharp\otimes I^\sharp = \F(M^\sharp)\;,$$
and a chain of isomorphisms:
\begin{align*}\E(M)^\sharp=\big(M^{[2]}_{*,1}(\cdot,k)\big)^\sharp&\simeq (M^{[2]}_{*,1})^\sharp(\cdot,k)\\
&\simeq (M^{[2]\,\sharp})_{*,1}(\cdot,k)\\
&\simeq (M^\sharp)^{[2]}_{*,1}(\cdot,k) = \E(M^\sharp)\;.
\end{align*}
In the chain of isomorphisms, the first isomorphism follows from the isomorphism of vector spaces $k^\vee\simeq k$, the second follows from the fact that duality preserves bidegrees, and the last one from the fact that duality of vector spaces commutes with direct sums.
\end{proof}

\begin{proposition}\label{second-adjunction}
The pair $(\E,\F)$ is an adjoint pair, i.e. we have an isomorphism, natural with respect to $M,N\in\PP$:
$$\alpha:Hom_\PP(\E(M),N)\simeq Hom_\PP(M,\F(N))\;. $$
\end{proposition}
\begin{proof}The adjunction isomorphism of proposition \ref{second-adjunction} is defined as the composite of the natural isomorphisms:
\begin{align*}
&Hom_\PP(\E(M),N)\simeq Hom_\PP(N^\sharp,\E(M)^\sharp)\simeq Hom_\PP(N^\sharp,\E(M^\sharp))\\&\simeq Hom_\PP(\F(N^\sharp),M^\sharp) \simeq Hom_\PP(\F(N)^\sharp,M^\sharp)\simeq Hom_\PP(M,\F(N))\;.
\end{align*}
\end{proof}
\subsection{The operators $X$ and $T$}

We first introduce the natural transformation $X:\E \to \E$.
We assume that $p\neq2$. For any $V\in \V$, let $U(\mathfrak{gl}(V\oplus k))$ denote the enveloping algebra of $\mathfrak{gl}(V\oplus k)$, and let $X_V \in U(\mathfrak{gl}(V \oplus k))$ be defined as follows. Fix a basis $V=\bigoplus_{i=1}^n ke_i$; this choice induces a basis of $V \oplus k$.  Let $x_{i,j} \in gl(V \oplus k)$ be the operator mapping $e_j$ to $e_i$ and $e_{\ell}$ to zero for all $\ell \neq j$.  Then define $$X_V=\sum_{i=1}^{n} x_{n+1,i}x_{i,n+1}-n.$$ 
The element $X_V$ does not depend on the choice of basis.  Indeed, let $C_V \in U(\mathfrak{gl}(V))$ denote the Casimir operator:
\begin{equation}
\label{Casimir}
C_V=\sum_{i\neq j}x_{i,j}x_{j,i}+\sum_{\ell=1}^nx_{\ell,\ell}^2.
\end{equation}
It's well known that $C_V$ does not depend on the choice of basis of $V$, and one computes (see Lemma 3.27 in \cite{HY11})
\begin{equation}
\label{EqXV}
C_{V\oplus k}-C_V=2X_V+\sum_{i=1}^n x_{i,i}-nx_{n+1,n+1}+x_{n+1,n+1}^2+2n.
\end{equation}
(Note this is where the hypothesis $p\neq 2$\ is used.)  Therefore $X_V$ also does not depend on the choice of basis of $V$.

The group $GL(V)\times GL(k)\subset GL(V\oplus k)$ acts on the Lie algebra $\mathfrak{gl}(V\oplus k)$ by the adjoint action, hence on the algebra $U(\mathfrak{gl}(V\oplus k)$. By (\ref{EqXV}) we have:  
\begin{lemma}
\label{lemma_X}
Let $V \in \V$.  Then $X_{V}$ commutes with $GL(V)\times GL(k)$, i.e.
$$
X_{V} \in U(gl(V \oplus k))^{GL(V)\times GL(k)}.
$$

\end{lemma}

The universal enveloping algebra $U(\mathfrak{gl}(V\oplus k))$ acts on $M(V\oplus k)$  via differentiation:
$$
d\pi_{M,V}:U(\mathfrak{gl}(V\oplus k)) \to End(M(V\oplus k)).
$$
\begin{example}\label{ex-tens}
If $M=I$ is the identity functor of $\V$, and $f\in \mathfrak{gl}(V\oplus k)$, then $d\pi_{I,V\oplus k}(f)=f$. More generally, if $d\ge 2$ and $M=\otimes^d$ is the $d$-th tensor product, then $d\pi_{\otimes^d,V\oplus k}$ sends $f\in\mathfrak{gl}(V\oplus k)$ onto the the element
$$\sum_{i=1}^d (1_{V\oplus k})^{\otimes i-1}\otimes f\otimes (1_{V\oplus k})^{\otimes d-i}\quad \in End((V\oplus k)^{\otimes d})\;.$$
\end{example}

The element $X_V$ acts on the vector space $M(V\oplus k)$ via $d\pi_{M,V}$, and we denote by $X_{M,V}$ the induced $k$-linear map:
$$X_{M,V}:M(V\oplus k)\to M(V\oplus k)\;.$$
By Lemma \ref{lemma_X}, $X_{M,V}$ is $GL(V)\times GL(k)$-equivariant. Thus, it restricts to the subspaces $\E(M)(V)$ of weight $1$ under the action of $\{1_V\}\times GL(k)$.  We denote the resulting map also by $X_{M,V}$:
$$
X_{M,V}:\E(M)(V) \to \E(M)(V).
$$ 

\begin{proposition}
\label{Prop_X}
The linear maps $X_{M,V}:\E(M)(V)\to \E(M)(V)$ are natural with respect to $M$ and $V$. Hence they define a morphism of functors
$$X:\E \to \E. $$
\end{proposition}
\begin{proof}
The action of $U(\mathfrak{gl}(V\oplus k))$ on $M(V\oplus k)$ is natural with respect to $M$. Hence the $k$-linear maps $X_{M,V}$ are natural with respect to $M$.

So it remains to check the naturality with respect to $V\in\V$. For this, it suffices to check that for all $M\in\PP$, and for all $f\in Hom(V,W)$, diagram (D) below is commutative.
$$\xymatrix{
M(V\oplus k)\ar[rr]^-{M(f\oplus 1_k)}\ar[d]^{X_{V,M}}&& M(W\oplus k)\ar[d]^{X_{W,M}}\\
M(V\oplus k)\ar[rr]^-{M(f\oplus 1_k)}&& M(W\oplus k)
}\qquad (D)$$

We observe that if diagram (D) commutes for a given strict polynomial functor $M$, then by naturality with respect to $M$, it also commutes for direct sums $M^{\oplus n}$, for $n\ge 1$, for the subfunctors $N\subset M$ and the quotients $M\twoheadrightarrow N$. But as we already explained in remark \ref{rk-tensors}, every functor $M\in\PP$ is a subquotient of a finite direct sum of copies of the tensor product functors $\otimes^d$, for $d\ge 0$. Thus, to prove naturality with respect to $V$, it suffices to check that diagram (D) commutes for $M=\otimes^d$ for all $d\ge 0$.

In the case of the tensor products $\otimes^d$ the action of $U(\mathfrak{gl}(V\oplus k)$ is explicitly given in example \ref{ex-tens}. Using this expression, a straightforward computation shows that diagram (D) is commutative in this case. This finishes the proof.  
\end{proof}

We next introduce a natural transformation $T:\E^2 \to \E^2$.  
Let $M\in\PP$ and $V\in\V$.  By definition, $$\E^2(M)=M^{[3]}_{*,1,1}(\cdot,k,k).$$ Consider the map $1_V \oplus \sigma:V \oplus k \oplus k \to V \oplus k \oplus k$ given by: $(v,a,b)\mapsto(v,b,a)$.  Applying $M^{[3]}$ to this map we obtain a morphism:
$$
T_{M,V}:M^{[3]}_{*,1,1}(V,k,k)\to M^{[3]}_{*,1,1}(V,k,k).
$$
\begin{lemma}
The linear maps $T_{M,V}:\E^2(M)(V)\to \E^2(M)(V)$ are natural with respect to $M$ and $V$. Hence they define a morphism of functors
$$T:\E^2 \to \E^2. $$
\end{lemma}

\begin{proof}
Clearly the maps $T_{M,V}$ are natural with respect to $M$.  Let $f:V\to W$ be a linear operator of vector spaces.  We need to show that the following diagram commutes:
\begin{equation*}
\xymatrix{\E^2(M)(V) \ar[rr]^{\E^2(M)(f)} \ar[d]_{T_{M,V}} && \E^2(M)(W) \ar[d]^{T_{M,W}}\\
\E^2(M)(V) \ar[rr]_{\E^2(M)(f)} &&  \E^2(M)(W)}
\end{equation*}
On the one hand, $\E^2(M)(f)$ is the restriction of $M^{[3]}(f\oplus 1_k\oplus 1_k)$ to the tri-degrees $(*,1,1)$.  On the other hand, $T_{M,V}$ is the restriction of $M^{[3]}(1_V\oplus \sigma)$ to the tri-degrees $(*,1,1)$.  Since $f\oplus 1_k\oplus 1_k$ clearly commutes with $1_V\oplus \sigma$, the above diagram commutes. 
\end{proof}

\subsection{The weight decomposition of $\PP$}

As part of the data of $\g$-categorification, we need to introduce a decomposition of $\PP$ indexed by the weight lattice $P$ of $\g$.  In this section we define such a decomposition via the blocks of $\PP$.

We begin by recalling some combinatorial notions.  For a nonnegative integer $d$, let $\wp_d$ denote the set of partitions of $d$.  A partition $\lambda$ is a \emph{$p$-core} if there exist no $\mu \subset \lambda$ such that the skew-partition $\lambda/\mu$ is a rim $p$-hook.  By definition, if $p=0$ then all partitions are $p$-cores.  Given a partition $\lambda$, we denote by $\tilde{\lambda}$ the $p$-core obtained by successively removing all rim $p$-hooks. The \emph{$p$-weight} of $\lambda$ is by definition the number $(|\lambda|-|\tilde{\lambda}|)/p$.  The notation $|\lambda|$ denotes the size of the partition $\lambda$.  Define an equivalence relation $\sim$ on $\wp_d$ by decreeing $\lambda \sim \mu$ if $\tilde{\lambda}=\tilde{\mu}$.

 Let $\lambda,\mu\in\wp_{d}$.  As a consequence of (11.6) in \cite{K03} we have
\begin{equation}\label{KleshEqua}
\tilde{\lambda}=\tilde{\mu} \Longleftrightarrow wt(\lambda)=wt(\mu).
\end{equation}
(See (\ref{weight}) for the definition of $wt(\lambda)$.)  Therefore we index the set of equivalence classes $\wp_d/\sim$ by weights in $P$, i.e. a weight $\omega \in P$ corresponds to a subset (possibly empty) of $\wp_d$.

Let $Irr\PP_d$ denote the set of simple objects in $\PP_d$ up to isomorphism.  This set is naturally identified with $\wp_d$.  We say two simple objects in $\PP_d$ are \emph{adjacent} if they occur as composition factors of some indecomposable object in $\PP_d$.    Consider the equivalence relation $\approx$  on $Irr\PP_d$ generated by adjacency.    Via the identification of $Irr\PP_d$ with $\wp_d$ we obtain an equivalence relation $\approx$ on $\wp_d$.

\begin{theorem}[Theorem 2.12, \cite{D}]
The equivalence relations $\sim$ and $\approx$ on $\wp_d$ are the same.
\end{theorem}

Given an equivalence class $\Theta \in Irr\PP_d/\approx$, the corresponding \emph{block} $\PP_\Theta \subset \PP_d$ is the subcategory of objects whose composition factors belong to $\Theta$.  The \emph{block decomposition} of $\PP$ is given by
$
\PP=\bigoplus\PP_{\Theta}
$, 
where $\Theta$ ranges over all classes in $Irr\PP_d/\approx$ and $d\geq 0$.
 
 By the above theorem and Equation (\ref{KleshEqua}), we can label the blocks  of $\PP_d$ by weights $\omega \in P$.  Moreover, by Equation (\ref{weight}), $wt(\lambda)$ determines the size of $\lambda$.  Therefore the block decomposition of $\PP$ can be expressed as:
$$\PP=\bigoplus_{\omega\in P} \PP_\omega. $$
The \emph{$p$-weight} of a block $\PP_\omega$ is the $p$-weight of $\lambda$, where $\omega=wt(\lambda)$.  This is well-defined since if $wt(\lambda)=wt(\mu)$ then $|\lambda|=|\mu|$ and $\tilde{\lambda}=\tilde{\mu}$, and hence the $p$-weights of $\lambda$ and $\mu$ agree.

\section{Categorification of Fock space}

In the previous section we defined all the data necessary to formulate the action $\g$ on $\PP$.  In this section we prove the main theorem:

\begin{theorem}
\label{mainthm}
Suppose $p\neq2$.  The category $\PP$ along with the data of adjoint functors $\E$ and $\F$, operators $X\in End(\E)$ and $T\in End(\E^2)$, and the weight decomposition $\PP=\bigoplus_{\omega\in P}\PP_{\omega}$ defines a $\g$-categorification which categorifies the Fock space representation of $\g$.  
\end{theorem}

\begin{remark}
The theorem is still true for $p=2$.  We only include this hypothesis for ease of exposition (one can prove the $p=2$ case using hyperalgebras instead of enveloping algebras).
\end{remark}

To prove this theorem we must show that the data satisfies properties (4)-(6), (8) of Definition \ref{DefinitionCategorification}, and that the resulting representation of $\g$ on $K(\PP)$ is isomorphic to the Fock space representation (property (7) already appears as Proposition \ref{first-adjunction}).

\subsection{The functors $\E_i$}
In this section we prove property (4) of Definition \ref{DefinitionCategorification}.
For all $a\in k$, and $M\in\PP$ we can form a nested collection of subspaces   of $\E(M)$, natural with respect to $M$:
$$0\subset \E_{a,1}(M) \subset \E_{a,2}(M)\subset \dots \subset \E_{a,n}(M)\subset \dots \subset \E(M) \;,$$
where $\E_{a,n}(M)$ is the kernel of $(X_M-a)^n:\E (M)\to \E(M)$. We define:
$$\E_a(M)=\bigcup_{n\ge 0} \E_{a,n}(M)\;.$$
Since the inclusions $\E_{a,n}(M)\subset \E_{a,n+1}(M)$ are natural with respect to $M$, the assignment $M\mapsto \E_a(M)$ defines a sub-endofunctor of $\E$.

\begin{lemma}\label{lm-aux}
The endofunctor $\E:\PP\to\PP$ splits as a direct sum of its subfunctors $\E_a$:
$$\E=\bigoplus_{a\in k} \E_a\;. $$
Moreover, for all $M\in\PP$ there exists an integer $N$ such that for all $n\ge N$, $\E_a(M)=\E_{a,n}(M)$.
\end{lemma}
\begin{proof}
The decomposition as a direct summand of generalized eigenspaces is standard linear algebra. The finiteness of the filtration $(\E_{a,n}(M))_{n\ge 0}$ follows from Corollary \ref{cor-finite-filtration}.
\end{proof}

\begin{proposition}\label{prop-decomp-E}Let $\lambda\in\wp$ be a partition of $d$ and set $W=W_{\lambda}$.
\begin{enumerate}
\item The  polynomial functor $\E(W)$ carries a Weyl filtration:
$$0=\E(W)^0\subset \E(W)^1\subset \dots \subset \E(W)^N=\E(W).$$
The composition factors which occur in this filtration are isomorphic to $W_\mu$ for all $\mu$ such that $\mu\longrightarrow \lambda$ and each such factor occurs exactly once.
\item The operator $X_W:\E(W)\to \E(W)$ preserves the filtration of $\E(W)$, and hence it acts on the associated graded object. 
\item Given $0\le i\le N-1$, set $j\in\Z/p\Z$ and $\mu\in\wp$ such that $\E(W)^{i+1}/\E(W)^i\simeq W_\mu$, and $\mu\xrightarrow[]{\;j\;}\lambda$. Then $X_W$ acts on $\E(W)^{i+1}/\E(W)^i$ by multiplication by $j$.
\end{enumerate}
In particular $\E_a=0$ for $a\not\in\Z/p\Z$, and hence 
$$\E=\bigoplus_{i\in\Z/p\Z}\E_i\;.$$
\end{proposition}
\begin{proof}
Theorem II.4.11 of  \cite{ABW} yields a filtration of the bifunctor $S_\lambda^{[2]}$ with associated graded object $\bigoplus_{\alpha\subset \lambda}S_\alpha\boxtimes S_{\lambda/\alpha}$. Here, $S_{\lambda/\alpha}\in \PP_{|\lambda|-|\alpha|}$ refers to the Schur functor associated to the skew partition $\lambda/\alpha$ and $S_\alpha\boxtimes S_{\lambda/\alpha}$ is the homogeneous bifunctor of bidegree $(|\alpha|,|\lambda|-|\alpha|)$, defined by
$(V,U)\mapsto S_\alpha(V)\otimes S_{\lambda/\alpha}(U)$.
Thus, $(S_\lambda^{[2]})_{*,1}$ has a filtration whose graded object is the sum of the $S_\alpha\boxtimes S_{\lambda/\alpha}$ with $|\lambda|=|\alpha|+1$. In this case, $S_{\lambda/\alpha}$ is the identity functor of $\V$ by definition. Thus, taking $U=k$, we get a filtration of $\E(S_\lambda)$ whose graded object is $\bigoplus S_\alpha$, for all $\alpha\rightarrow \lambda$. The first part of the proposition follows by duality $^\sharp$. (For an alternative proof based on \cite{Martin} and \cite[Thm. 8.1.1]{GW}, see \cite[Lemma A.3]{HY11}.)

For any $V\in\V$, by (\ref{EqXV}) the map $X_{W,V}$ preserves the filtration of $GL(V)$-modules:
$$0=\E(W)^0(V)\subset \E(W)^1(V)\subset \dots \subset \E(W)^N(V)=\E(W)(V).$$ Indeed, since Weyl modules are highest weight modules, $C_{V\oplus k}$ acts on $W(V\oplus k)$ by scalar, and $C_V$ acts on the factors of the filtration by scalar as well.  Therefore $X_W$ preserves the filtration of $\E(W)$, proving the second part of the proposition.
 
Finally, let $\mu$ and $j$ be chosen as in the third part of the proposition.  By Lemma 5.7(1) in \cite{HY11}, for any $V\in \V$, $X_{W,V}$ acts by $j$ on $\E(W)^{i+1}(V)/\E(W)^i(V)$.  Therefore $X_W$ acts on $\E(W)^{i+1}/\E(W)^i$ also by $j$.

\end{proof}

By the adjunction of $\E$ and $\F$ and the Yoneda Lemma, the operator $X\in End(\E)$ induces an operator $X^{\circ}\in End(\F)$.  The generalized eigenspaces of this operator produce subfunctors $\F_a$ of $\F$, which, by general nonsense, are adjoint to $\E_a$.  Therefore we have decompositions
$$
\E=\bigoplus_{i\in\Z/p\Z}\E_i,
\F=\bigoplus_{i\in\Z/p\Z}\F_i.
$$

\subsection{The action of $\g$ on $K(\PP)$}
In this section we prove property (5) of Definition \ref{DefinitionCategorification}. The functors $\E_i,\F_i$, being exact functors, induce linear operators
 $$
 [\E_i],[\F_i]:K(\PP) \to K(\PP)
 $$
for all $i\in\Z/p\Z$.  Define a map $\varkappa: \g\to End(K(\PP))$ by $e_i\mapsto [\E_i]$ and $f_i\mapsto [\F_i]$.  Let $\Psi:K(\PP)\to\FF$ be given by $\Psi([W_{\lambda}])=v_{\lambda}$.    
\begin{proposition}
The map $\varkappa$ is a representation of $\g$ and $\Psi$ is an isomorphism of $\g$-modules.
\end{proposition}

\begin{proof}
By Corollary \ref{LemmaKgrp} $\Psi$ is a linear isomorphism.  By Proposition \ref{prop-decomp-E}, 
$$
[\E_i]([W_{\lambda}])=\sum_{\xymatrix{\mu \ar[r]^i& \lambda}}[W_{\mu}].
$$ Therefore $\Psi$ intertwines the operators $e_i$ and $[\E_i]$, i.e. $\Psi \circ [\E_i]=e_i \circ \Psi$.  Consider the bilinear form on $K(\PP)$ given by:
$$
\left< M,N\right>=\sum_{i\geq0}(-1)^idimExt^i(M,N)
$$
By adjunction $[\E_i]$ and $[\F_i]$ are adjoint operators with respect to $\left<\cdot,\cdot\right>$, and by Theorem \ref{portemanteau}(iv), $\left< W_{\lambda},S_{\mu}\right>=\delta_{\lambda\mu}$.  Therefore
 $$
[\F_i]([S_{\lambda}])=\sum_{\xymatrix{\lambda \ar[r]^i& \mu}}[S_{\mu}].
$$     
Hence, by Corollary \ref{lm-dual}, $\Psi$ also intertwines the operators $f_i$ and $[\F_i]$.  Both claims of the proposition immediately follow.
\end{proof}    

\subsection{Chevalley functors and weight decomposition of $\PP$}

In this section we prove property (6) of Definition \ref{DefinitionCategorification}.

\begin{proposition}
Let $\omega\in P$.  For every $i\in\Z/p\Z$, the functors $\E_i,\F_i:\PP\to\PP$ restrict to
$
\E_i:\PP_{\omega}\to\PP_{\omega+\alpha_i}$ and $
\F_i:\PP_{\omega}\to\PP_{\omega-\alpha_i}.
$
\end{proposition}

\begin{proof}

We prove that $\E_i(\PP_\omega)\subset\PP_{\omega+\alpha_i}$ (the proof for $\F_i$ being entirely analogous).  Since $\E_i$ is exact it suffices to prove that if $L_\lambda\in\PP_\omega$ then $\E_i(L_\lambda)\in\PP_{\omega+\alpha_i}$.  Then, by the same idea as used in the proof of Lemma \ref{LemmaKgrp}, it suffices to show that if $W_\lambda\in\PP_\omega$ then $\E_i(W_\lambda)\in\PP_{\omega+\alpha_i}$. By Proposition \ref{prop-decomp-E}, $\E_i(W_\lambda)$ has a Weyl filtration with factors all of the form $W_\mu$, where $\xymatrix{\mu \ar[r]^i& \lambda}$.  But then $\mu\in\omega+\alpha_i$, so $W_\mu\in\PP_{\omega+\alpha_i}$.  Therefore $\E_i(W_\lambda)\in\PP_{\omega+\alpha_i}$.   
\end{proof}

\subsection{The degenerate affine Hecke algebra action on $\E^n$}

In this section we prove property (8) of Definition \ref{DefinitionCategorification}.

\begin{proposition}\label{prop-Hecke}
The assignments
\begin{equation*}
y_i \mapsto \E^{n-i}X\E^{i-1} \text{ for }1 \leq i \leq n,
\end{equation*}
\begin{equation*}
\tau_i \mapsto \E^{n-i-1}T\E^{i-1} \text{ for }1 \leq i \leq n-1.
\end{equation*}
define an action of $DH_n$ on $End(\E^n)$.
\end{proposition} 
\begin{proof}
By definition, $\E^n(M)(V)$ is the subspace of $M(V\oplus k^{n})$ formed by the vectors of weight $\varpi_{n}=(1,1,\dots,1)$ for the action of $GL(k)^{\times n}$.  Here $GL(k)^{\times n}$ acts via the composition:
$$GL(k)^{\times n}=1_V\times GL(k)^{\times n}\subset GL(V\oplus k^{n})\xrightarrow[]{\pi_{M,V\oplus k^{n}}}GL(M(V\oplus k^{n}))\;. $$ 

The map $(\tau_{n-i})_{M,V}$ is equal to the restriction of $M(t_i)$ to $\E^n(M)(V)$, where $t_i:V\oplus k^{n}\to V\oplus k^{n}$ maps $(v,x_1,\dots,x_n)$ to $(v,x_1,\dots, x_{i+1},x_i,\dots,x_n)$. To check that the $\tau_i$ define an action of $\Z\mathfrak{S}_n$ on $\E^n$, we need to check that the $(\tau_i)_{M,V}$ define an action of the symmetric group on $\E^n(M)(V)$.  By Remark \ref{rk-tensors} it suffices to check this for $M=\otimes^d$, and this is a straightforward computation. Moreover, it is also straightforward from the definition that the $y_i$ commute with each other. Thus they define an action of the polynomial algebra $\Z[y_1,\dots,y_n]$ on $\E^n$. Similarly, $\tau_i$ and $y_j$ commute with each other if $|i-j|\ge 1$.

So, to obtain the action of the Hecke algebra on $End(\E^n)$, it remains to show that $\tau_iy_i-y_{i+1}\tau_i=1$ (see Remark \ref{Pedantry}).
This will be proved by showing the following identity in $End(\E^2)$:
\begin{equation}
\label{EqHecke}
 T\circ \E X - X\E \circ T=1. 
\end{equation}
To check (\ref{EqHecke}), it suffices to check that for all $M\in\PP$ and all $V\in\V$, 
\begin{equation}
\label{EqD}
T_{M,V}\circ\E(X_M)_V-X_{\E(M)}\circ T_{M,V}=1_{\E^2(M)(V)}
\end{equation}


If (\ref{EqD}) holds for $M\in\PP$, then by naturality with respect to $M$, it also holds for direct sums $M^{\oplus n}$, for subfunctors $N\subset M$, and quotients $M\twoheadrightarrow N$. By Remark \ref{rk-tensors}, every functor $M\in\PP$ is a subquotient of a finite direct sum of copies of the tensor product functors $\otimes^d$, for $d\ge 0$. Thus, it suffices to check that Equation (\ref{EqD}) holds for $M=\otimes^d$ for all $d\ge 0$.

Let $M=\otimes^d$ and let $V\in\V$.  Choose a basis $(e_1,\dots,e_n)$ of $V$. We naturally extend this to a basis $(e_1,\dots,e_{n+2})$ of $V\oplus k\oplus k$. 
By definition, $\E^2(\otimes^d)(V)$ is the subspace of $(V\oplus k\oplus k)^{\otimes d}$ spanned by the vectors of the form $e_{i_1}\otimes \dots \otimes e_{i_d}$, where exactly one of the $e_{i_k}$ equals $e_{n+1}$ and exactly one of the $e_{i_k}$ equals $e_{n+2}$. Let us fix a vector $\xi=e_{i_1}\otimes \dots \otimes e_{i_d}$ with $e_{n+1}$ in $a$-th position and $e_{n+2}$ in $b$-th position.  We will show that Equation (\ref{EqD}) holds for $\xi$.

First note that $T_{M,V}(\xi)=e_{i_{(a b)(1)}}\otimes \dots \otimes e_{i_{(a b)(d)}}$, where $(ab)$ denotes the transposition of $\mathfrak{S}_d$ which exchanges $a$ and $b$.  
Then
\begin{eqnarray*}
(X\E)_{M,V}\circ T_{M,V}(\xi)&=&\left(\sum_{j=1}^nx_{n+1,j}x_{j,n+1}-n\right).(e_{i_{(a b)(1)}}\otimes \dots \otimes e_{i_{(a b)(d)}}) \\ &=& \sum_{\ell \neq a,b} e_{i_{(\ell b a)(1)}}\otimes \dots \otimes e_{i_{(\ell b a)(d)}}.
\end{eqnarray*}
%

Now we compute the other term on the left hand side of (\ref{EqD}).  Then,
\begin{eqnarray*}
T_{M,V}\circ (\E X)_{M,V}(\xi) &=& T_{M,V} \circ \left(\sum_{j=1}^{n+1}x_{n+2,j}x_{j,n+2}-(n+1)\right)(\xi) \\ &=& \sum_{\ell \neq a,b} e_{i_{(\ell b a)(1)}}\otimes \dots \otimes e_{i_{(\ell b a)(d)}} + \xi. 
\end{eqnarray*}
Therefore (\ref{EqD}) holds.
\end{proof}

This completes the proof of Theorem \ref{mainthm}.

\section{Remarks}

We conclude the paper by mentioning briefly some consequences of the $\g$-categorification on $\PP$.

\subsection{Derived equivalances}
For this discussion we focus on the case  $p=char(k)>0$.  The main motivation for Chuang and Rouquier's original work on categorification was to prove Broue's abelian defect conjecture for the symmetric groups, which can be reduced to showing that any two blocks of symmetric groups of the same p-weight are derived equivalent \cite{CR}. Their technique applies to the setting of $\mathfrak{sl}_2$-categorifications.  Since  for every simple root $\alpha$ of $\g$ there is a corresponding root subalgebra of $\g$ isomorphic to $\mathfrak{sl}_2$, we've in fact defined a family of $\mathfrak{sl}_2$-categorifications on $\PP$.  To each of these categorifications we can apply the Chuang-Rouqueir machinery.   

Let $W^{aff}=\Sy_p \ltimes Q$ denote the affine Weyl group associated to $\g$, acting on $P$ in the usual way.  By \cite[Section 12]{Kac},  any weight $\omega$ appearing in the weight decomposition of Fock space is of the form $\sigma(\omega_0)-\ell \delta$, where $\sigma \in W^{aff}$ and $\ell \geq 0$.  By Proposition 11.1.5 in \cite{K03},  $\ell $ is exactly the p-weight of the corresponding block.  Therefore the weights of any two blocks are conjugate by some element of affine Weyl group if and only if they have the same p-weight.  By Theorem 6.4 in \cite{CR} we obtain 

\begin{theorem}
If two blocks of $\PP$ have the same p-weight, then they are derived equivalent.
 \end{theorem}

\subsection{Misra-Miwa crystal}

We can also apply the theory of $\g$-categorification to crystal basis theory. The crystal structure is a combinatorial structure associated to integrable representations of Kac-Moody algebras, introduced originally by Kashiwara via the theory of quantum groups.  From Kashiwara's theory one can construct a canonical basis for the corresponding representations, which agrees with Lustig's canonical basis of geometric origins.

Loosely speaking, the crystal structure of an integrable representation of some Kac-Moody algebra consists of a set $\mathbb{B}$ in bijection with a basis of the representation, along with Kashiwara operators $\tilde{e}_i,\tilde{f}_i$ on $\mathbb{B}$ indexed by the simple roots of the Kac-Moody algebra, along with further data.  For a precise definition see \cite{Ka}.  

From the $\g$-categorification on $\PP$ we can recover the crystal structure of Fock space as follows.  For the set $\mathbb{B}$ we take $Irr\PP\subset K(\PP)$, the set of equivalence classes of simple objects.  We construct Kashiwara operators on $Irr\PP$ by composing the Chevalley functors with the $socle$ functor:
$$
\tilde{e}_i,\tilde{f}_i=[socle\circ E_i],[socle\circ F_i]:Irr\PP \to Irr\PP.
$$    
The other data defining a crystal structure can also be naturally obtained.  In Section 5.3 of \cite{HY11} it is shown that this data agrees with the crystal of $\FF$ originally discovered by Misra and Miwa \cite{MM}.

Jiuzu Hong,
School of Mathematical Sciences
Tel Aviv University,
Tel Aviv
69978, Israel.
\texttt{hjzzjh@gmail.com}
\medskip\\
Antoine Touz\'e, LAGA Institut Galil\'ee, Universit\'e Paris 13, 99, Av. J-B Cl\'ement 93430 Velletaneuse, France.\\  \texttt{touze@math.univ-paris13.fr }
\medskip\\
Oded Yacobi,
Department of Mathematics,
University of Toronto,
Toronto, ON, M5S 2E4
Canada.\\  \texttt{oyacobi@math.toronto.edu}


\begin{thebibliography}{99}

\bibitem[ABW]{ABW} K.~Akin, D.A.~Buchsbaum, J.~Weyman, Schur functors and Schur complexes, Adv. Math. (1982), 207--278.

\bibitem[CR]{CR} Chuang, J.; Rouquier, Derived equivalences for symmetric groups and $sl_2$-categorification. Ann. of Math. (2) 167 (2008), no. 1, 245-298.

\bibitem[D]{D} Donkin, Stephen On Schur algebras and related algebras. II. J. Algebra 111 (1987), no. 2, 354-364.

\bibitem[FFSS]{FFSS} V.~Franjou, E.~Friedlander, A.~Scorichenko, A.~Suslin, General linear and functor cohomology over finite fields,   Ann. of Math. (2)  150  (1999),  no. 2, 663--728.
\bibitem[FS]{FS} E.~Friedlander, A.~Suslin, Cohomogy of finite group schemes over a field,
 Invent. Math. 127 (1997), 209--270.


\bibitem[GW]{GW} R.~Goodman, N.R.~Wallach, Symmetry, representations, and invariants. Graduate Texts in
Mathematics, 255. Springer, Dordrecht, 2009.

\bibitem[HY]{HY11} Hong, J., Yacobi, O.  Polynomial representations of general linear groups and categorifications of Fock space, preprint, {\tt arXiv:1101.2456}.

\bibitem[HY2]{HY2} Hong, J., Yacobi, O. Polynomial functors and categorifications of Fock space II: Schur-Weyl duality, preprint.

 
\bibitem[J]{Jantzen} Jantzen, J. C. Representations of algebraic groups. Second edition. Mathematical Surveys and Monographs, 107. American Mathematical Society, Providence, RI, 2003.

\bibitem[Kac]{Kac}  Kac, Victor G. Infinite-dimensional Lie algebras. Third edition. Cambridge University Press, Cambridge, 1990. 

\bibitem[Kas]{Ka} Kashiwara, Masaki, Crystalizing the q-analogue of universal enveloping algebras. 
Comm. Math. Phys. 133 (1990), no. 2, 249-260. 


\bibitem[KL]{KL}  M. Khovanov, A. Lauda, A categorification of quantum sl(n). Quantum Topol. 1 (2010), no. 1, 1-92. 
 
\bibitem[Kl]{K03} Kleshchev, A. Linear and projective representations of symmetric groups. Cambridge Tracts in Mathematics, 163. Cambridge University Press, Cambridge, 2005.

\bibitem[Ku]{K} N.~Kuhn, Rational cohomology and cohomological stability in generic representation theory. Amer. J. Math. 120 (1998) 1317--1341.

\bibitem[LLT]{LLT} Lascoux, A.; Leclerc, B.; Thibon, J. Hecke algebras at roots of unity and crystal bases of quantum affine algebras. Comm. Math. Phys. 181 (1996), no. 1, 205-263.

\bibitem[Mac]{Mac} Macdonald, I. G. Symmetric functions and Hall polynomials. Second edition. With contributions by A. Zelevinsky. Oxford Mathematical Monographs. Oxford Science Publications. The Clarendon Press, Oxford University Press, New York, 1995.
\bibitem[Mar]{Martin} S.~Martin, Schur algebras and representation theory, Cambridge tracts in mathematics 112
\bibitem[Mat]{Mathieu} O.~Mathieu, Filtrations of G-modules. Ann. Sci. \'Ecole Norm. Sup. (4) 23 (1990), no. 4, 625-644.
\bibitem[MM]{MM} Misra, K.C, Miwa, T.: Crystal base of the basic representation of $U_{q}(\hat{\mathfrak{sl}}_n)$. Commun. Math. Phys 134, 79-88 (1990)
\bibitem[R]{R} Rouquier, R., 2-Kac-Moody algebras, preprint,  arXiv:0812.5023
\bibitem[SFB]{SFB} A.~Suslin, E.~Friedlander, C.~Bendel, 
 Infinitesimal $1$-parameter subgroups and cohomology.  J. Amer. Math. Soc. 10  (1997),  no. 3, 693--728.
\bibitem[T1]{T2}A. Touz\'e, Cohomology of classical algebraic groups from the functorial viewpoint, Adv. Math. 225 (2010), no. 1, 33--68.
\bibitem[T2]{T1}A. Touz\'e, Koszul duality and derivatives of non-additive functors, {\tt arXiv:1103.4580}.
\bibitem[TvdK]{TVdK} A. Touz\'e, W.~van der Kallen, Bifunctor cohomology and cohomological finite generation for reductive groups, Duke Math. J.  151  (2010),  no. 2, 251--278.

\end{thebibliography}
\end{document}